\documentclass[a4paper,11pt]{article}
\usepackage{amsmath}
\usepackage{amssymb}
\usepackage{setspace}
\usepackage{fullpage}
\usepackage{bbm}
\usepackage{color}
\usepackage{amsthm}
\usepackage{pbox}
\newtheorem{theorem}{Theorem}[section]
\newtheorem*{mmtheorem}{Theorem \ref{r+1normal}}
\newtheorem{lemma}[theorem]{Lemma}

\newtheorem{corollary}[theorem]{Corollary}

\theoremstyle{definition}

\newtheorem*{definition}{Definition}

\newtheorem*{remark}{Remark}

\usepackage{pdfsync}
\def\PG{\mathrm{PG}} \def\AG{\mathrm{AG}}

\def\PGL{\mathrm{PGL}}

\def\A{\mathcal{A}}  
\def\D{\mathcal{D}}  
 \def\K{\mathcal{K}}
\def\L{\mathcal{L}} \def\M{\mathcal{M}} 
 \def\P{\mathcal{P}}
\def\R{\mathcal{R}} \def\S{\mathcal{S}}
\def\T{\mathcal{T}}
\def\Mb{\mathbf{M}}
\def\NN{\mathbb{N}}

\def\F{\mathbb{F}}
\def\Fq{\mathbb{F}_q}
\def\Fqn{\mathbb{F}_{q^n}}

\title{A geometric characterisation of Desarguesian spreads}
\author{ Sara Rottey \and John Sheekey\thanks{This author is supported by the Fund for Scientific Research Flanders (FWO -- Vlaanderen).}}
\date{}
\begin{document}
\maketitle
\begin{abstract}
We provide a characterisation of $(n-1)$-spreads in $\PG(rn-1,q)$, $r>2$, that have $r$ normal elements in general position. In the same way, we obtain a geometric characterisation of Desarguesian $(n-1)$-spreads in $\PG(rn-1,q)$, $r>2$.
\end{abstract}

\section{Introduction}

In this paper, we study spreads in finite projective spaces. An $(n-1)$-spread of $\Pi=\PG(N-1,q)$ is a set of $(n-1)$-dimensional subspaces partitioning the points of $\Pi$. It is well-known since Segre \cite{Segre} that an $(n-1)$-spread exists in $\PG(N-1,q)$ if and only if $n$ divides $N$. The ``only if'' direction follows from a counting argument, while the other direction follows by Segre's construction of {\em Desarguesian spreads}.

Desarguesian spreads play an important role in finite geometries; for example in field reduction and linear (blocking) sets \cite{MichelGeertruiFieldReduction}, eggs and TGQ's \cite{BaderLunardon}. Moreover, every $(n-1)$-spread canonically defines an incidence structure, which is a translation $2-(q^{rn}, q^n, 1)$ design with parallelism. When the spread is Desarguesian, the corresponding design is isomorphic to the affine space $\AG(r, q^n)$. 

Notwithstanding their importance, few geometric characterisations of Desarguesian are known.
A geometric characterisation of Desarguesian $(n-1)$-spreads in $\PG(rn-1,q)$ was obtained by Beutelspacher and Ueberberg in {\rm \cite[Corollary]{BeutelspacherUeberberg}} by considering the intersection of spread elements with all $r(n-1)$-subspaces.
However, the
two most famous and important characterisations of Desarguesian spreads arise
from their correspondence with regular (for $r=2$ and $q > 2$) and normal spreads (for $r > 2$).
An $(n-1)$-spread in $\PG(2n-1, q)$ is said to be {\em regular} if it contains every regulus defined by any three of its elements. When $q>2$, an $(n-1)$-spread in $\PG(2n-1, q)$ is regular if and only if it is Desarguesian.
An $(n-1)$-spread in $\PG(rn-1,q)$ is said to be {\em normal} if it induces a spread in the $(2n-1)$-space spanned by any two of its elements. It is known that an $(n-1)$-spread in $\PG(rn-1, q)$, $r>2$, is normal if and only if it is Desarguesian (see \cite{BarlottiCofman}).

We will focus on the normality of Desarguesian spreads, by introducing the notion of a {\em normal element} of a spread. We say that an element $E$ of an $(n-1)$-spread $\S$ of $\Pi=\PG(rn-1,q)$ is {\em normal} if $\S$ induces a spread in the $(2n-1)$-space spanned by $E$ and any other element of $\S$. 
Clearly, by definition, a spread is normal if and only if all its elements are normal. The motivation for this article consists of the following questions:
\begin{itemize}
\item[]How many normal elements does a spread need, to ensure that it is normal/Desarguesian?
\item[]Can we characterise a spread given the configuration of its normal elements?
\end{itemize}

We will give answers to both questions; the answer to the first question is easily stated in the following theorem.

\begin{mmtheorem}
Consider an $(n-1)$-spread $\S$ in $\PG(rn-1,q)$, $r>2$.
If $\S$ contains $r+1$ normal elements in general position, then $\S$ is a Desarguesian spread.
\end{mmtheorem}

This paper is organised as follows. In Section \ref{PrelimQuasi}, we introduce the necessary preliminaries. In Section \ref{r normal elements} we obtain a characterisation of $(n-1)$-spreads in $\PG(rn-1,q)$ having $r$ normal elements in general position. We will see in Section \ref{r+1 normal elements} that for some $n$ and $q$ these spreads must be Desarguesian. Moreover, we obtain a characterisation of Desarguesian spreads as those having at least $r+1$ normal elements in general position. In Section \ref{3 on a line}, we consider spreads containing normal elements, not in general postion, but contained in the same $(2n-1)$-space.
In Section \ref{lessthanr} we consider examples of spreads in $\PG(rn-1,q)$ containing less than $r$ normal elements, and we characterise the spreads with $r-1$ normal elements in general position. Finally in Section \ref{sperner} we formulate our results in terms of the designs corresponding to the spreads.

\section{Preliminaries}\label{PrelimQuasi}
In this paper, all considered objects will be finite. Denote the $m$-dimensional projective space over the finite field $\F_q$ with $q$ elements, $q=p^h$, $p$ prime, by $\PG(m,q)$.

\subsection{Desarguesian spreads and choosing coordinates}

A {\em Desarguesian spread} of $\PG(rn-1,q)$ can be obtained by applying {\em field reduction} to the points of $\PG(r-1,q^n)$. The underlying vector space of the projective space $\PG(r-1,q^n)$ is $V(r,q^n)$; if we consider $V(r,q^n)$ as a vector space over $\F_q$, then it has dimension $rn$, so it defines $\PG(rn-1,q)$. In this way, every point $P$ of $\PG(r-1,q^n)$ corresponds to a subspace of $\PG(rn-1,q)$ of dimension $(n-1)$ and it is not hard to see that this set of $(n-1)$-spaces forms a spread of $\PG(rn-1,q)$, which is called a Desarguesian spread. For more information about field reduction and Desarguesian spreads, we refer to \cite{MichelGeertruiFieldReduction}.

\begin{definition}
An element $E$ of an $(n-1)$-spread $\S$ of $\PG(rn-1,q)$ is a {\em normal element} of $\S$ if, for every $F \in \S$, the $(2n-1)$-space $\langle E, F\rangle$ is partitioned by elements of $\S$. Equivalently, this means that $\S/E$ defines an $(n-1)$-spread in the quotient space $\PG((r-1)n-1,q)\cong \PG(rn-1,q)/E$.
\end{definition}
A spread is called {\em normal} when all its elements are normal elements.
An $(n-1)$-spread of $\PG(rn-1,q)$, $r>2$, is normal if and only if it is Desarguesian (see \cite{BarlottiCofman}).

We will use the following coordinates. The point $P$ of $\PG(r-1,q^n)$, defined by the vector $(a_1,a_2,\ldots,a_r)\in (\F_{q^n})^r$, is denoted by $(a_1,a_2,\ldots,a_r)_{\F_{q^n}}$, reflecting the fact that every $\F_{q^n}$-multiple of $(a_1,a_2,\ldots,a_r)$ gives rise to the point $P$.
We can identify the vector space $\F_{q}^{nr}$ with $(\F_{q^n})^r$, and hence write every point of $\PG(rn-1,q)$ as $(a_1,\ldots,a_r)_{\F_q}$, with $a_i \in \F_{q^n}$.

When applying field reduction, a point $(a_1,\ldots,a_r)_{\F_{q^n}}$ in $\PG(r-1,q^n)$ corresponds to the $(n-1)$-space $$\{  ( a_1 x ,\ldots,  a_r x)_{\F_q} \mid x \in \F_{q^n}\}$$ of $\PG(rn-1,q)$.

Moreover, every $(n-1)$-space in $\PG(rn-1,q)$ can be represented in the following way. Let $a_1,\ldots,a_r$ be $\Fq$-linear maps from $\Fqn$ to itself. Then the set $$\{(a_1(x),\ldots,a_r(x))_{\Fq} \mid x \in \Fqn\}$$ corresponds to an $(n-1)$-space of $\PG(rn-1,q)$. When choosing a basis for $\F_{q^n}\cong\F_q^n$ over $\F_q$, the $\F_q$-linear map $a_i, i=1,\ldots,r$, is represented by an $n\times n$-matrix $A_i, i=1,\ldots,r,$ over $\F_q$ acting on row vectors of $\F_q^n$ from the right. We abuse notation and write the corresponding $(n-1)$-space of $\PG(rn-1,q)$ as
\[
(A_1,\ldots,A_r) := \{(x A_1,\ldots,x A_r)_{\Fq} \mid x \in \F_q^n\}.
\]

\begin{definition}
We call a point set of $\PG(k-1,q)$, a point set {\em in general position}, if every subset of $k$ points spans the full space. A {\em frame} of $\PG(k-1,q)$ is a set of $k+1$ points in general position.
Equivalently, a set of $(n-1)$-spaces in $\PG(kn-1,q)$ such that any $k$ span the full space, is called a set of $(n-1)$-spaces {\em in general position}.\end{definition}

Note that for any set $\K$ of $k+1$ $(n-1)$-spaces $S_i$, $i=0,\ldots,k$, of $\PG(kn-1,q)$, $k>2$, in general position, there exists a field reduction map $\mathcal{F}$ from $\PG(k-1,q^n)$ to $\PG(kn-1,q)$, such that $\K$ is contained in the Desarguesian spread of $\PG(kn-1,q)$ defined by $\mathcal{F}$. This means $\K$ is the field reduction of a frame in $\PG(k-1,q^n)$. This is equivalent with saying that we can choose coordinates for $\PG(kn-1,q)$ such that
\begin{align*}
\K &= \left\{  S_0= (I, I,I, \ldots,I, I) , S_1= (I, 0,0, \ldots,0, 0),  \right. \\
& \hspace{1cm}\left. S_2= (0,I, 0, \ldots,0, 0),
\ldots,  S_k= (0,0,0, \ldots, 0,I)\right\},
\end{align*}
where $I$ denotes the identity map or the identity matrix.

\subsection{Quasifields and spread sets}\label{quasifieldsection}
Andr\'e \cite{Andre} and Bruck and Bose \cite{Br} obtained that finite translation planes and $(n-1)$-spreads of $\PG(2n-1,q)$ are equivalent objects (this connection and a more general connection with translation Sperner spaces will be considered in Section \ref{sperner}).
We will now consider their correspondence with finite quasifields and matrix spread sets.
For a more general overview than given here, we refer to \cite{BrB1966,Dembowski,HandbookTranslationPlanes}.

\begin{definition}
A {\em finite (right) quasifield} $(Q,+,*)$ is a structure, where $+$ and  $*$ are binary operations on $Q$, satisfying the following axioms:
\begin{enumerate}\setlength{\itemsep}{-1pt}
\item[$(i)$] $(Q,+)$ is a group, with identity element 0,
\item[$(ii)$] $(Q_0=Q \setminus \{0\},*)$ is a multiplicative loop with identity 1, i.e.\
$\forall a\in Q: 1* a= a* 1=a$ and $\forall a,b\in Q_0: a*x=b \mbox{ and } y*a=b$ have unique solutions $x,y\in Q$,
\item[$(iii)$] right distributivity: $\forall a,b,c \in Q: (a+b) * c=a * c+ b* c$,
\item[$(iv)$] $\forall a,b,c \in Q, a\neq b: x*a=x*b+c$ has a unique solution $x\in Q$.
\end{enumerate}
\end{definition}
In this paper, we will omit the term finite.


\begin{definition}
The {\em kernel} $K(Q)$ of a quasifield $(Q,+,*)$ is the set of all $ k \in Q$ satisfying
 \begin{align*}
 \forall x,y \in Q: \, \, & k*(x*y)=(k*x)*y, \,\mbox{ and }\\
 \forall x,y \in Q: \, \, & k*(x+y)=k*x+k*y.
 \end{align*}
Note that the kernel $K(Q)$ of a quasifield $Q$ is a field.
\end{definition}\index{kernel of a quasifield}

\begin{definition}
A {\em (matrix) spread set} is a family $\mathbf{M}$ of $q^n$ $n\times n$-matrices over $\F_q$ such that, for every two distinct $A,B\in \mathbf{M}$, the matrix $A-B$ is non-singular.
\end{definition}
Now, denote the points of $\PG(2n-1,q)$ by $\{(x,y)_{\F_q} \mid x,y\in\F_q^n, (x,y)\neq(0,0)\}$.
Given a spread set $\mathbf{M}$, consider
$$\S(\mathbf{M})=\left\{ E_A \mid A\in \mathbf{M}\right\} \cup \left\{ E_\infty \right\},$$
where
$$E_A=(I, A)=\{(x, xA)_{\F_q}\mid x\in\F_q^n\} \,\, \mbox{ and }\, \,  E_\infty=(0,I) =\{(0,x)_{\F_q} \mid x\in\F_q^n\}.$$

One can check that $\S(\mathbf{M})$ is an $(n-1)$-spread of $\PG(2n-1,q)$. Moreover, every $(n-1)$-spread of $\PG(2n-1,q)$ is equivalent (under a collineation of the space) to a spread of the form $\S(\mathbf{M})$, for some spread set $\mathbf{M}$, such that the zero matrix $0$ and identity matrix $I$ are both contained in $\mathbf{M}$ (see \cite[Section 5]{Br}).

Consider now the vector space $V=V(n,q)$ and take a non-zero vector $e$ of $V$. For every vector $y\in V$, there exists a unique matrix $M_y\in\mathbf{M}$ such that $y=eM_y$. Define multiplication $*$ in $V$ by
$$x*y=x M_y.$$

By \cite[Section 6]{Br}, using this multiplication and the original addition, $V$ becomes a right quasifield $\mathcal{Q}_e(\mathbf{M})=(V,+,*)$. Conversely, given a quasifield $\mathcal{Q}$ with multiplication $*$ on $V$ and $\F_q$ in its kernel, we can define a spread set $\Mb(\mathcal{Q}) = \{M_y \mid y \in V\}$, where $M_y$ is defined by $\forall x \in V: x M_y = x*y$. Clearly, $\mathcal{Q}=\mathcal{Q}_e(\mathbf{M}(\mathcal{Q}))$.

\begin{definition}
A {\em semifield} is a right quasifield also satisfying left distributivity.\index{semifield}
A {\em (right) nearfield} $\mathcal{Q}$ \index{nearfield} is a (right) quasifield that satisfies associativity for multiplication, i.e.\ $\forall a,b,c \in \mathcal{Q}: (a* b)* c = a *(b* c)$.
\end{definition}
Note that a quasifield which is both a semifield and a nearfield is a {\em (finite) field}.

\begin{theorem}{\rm \cite[Section 11]{BrB1966}}\label{SemiNear}

The quasifield $\mathcal{Q}_e(\mathbf{M})$ is a nearfield if and only if $\mathbf{M}$ is closed under multiplication.

The quasifield $\mathcal{Q}_e(\mathbf{M})$ is a semifield if and only if $\mathbf{M}$ is closed under addition.
\end{theorem}

We omit the correspondence to translation planes and conclude with the following connections (see \cite{Dembowski,HandbookTranslationPlanes}).

\begin{table}[h!]
\begin{tabular}{lll}
$\S$ is a {\em nearfield spread} &$\Leftrightarrow$& $\S\cong\S(\Mb)$ with $\mathbf{M}$ closed under multiplication;\\
$\S$ is a {\em semifield spread} &$\Leftrightarrow$&  $\S\cong\S(\Mb)$ with $\mathbf{M}$ closed under addition;\\
$\S$ is a Desarguesian spread &$\Leftrightarrow$& $\S\cong\S(\Mb)$ with $\mathbf{M}$ closed under multiplication and addition.
\end{tabular}
\end{table}

Equivalent to the previous, we say $\mathbf{M}$ is a {\em nearfield spread set}, respectively {\em semifield spread set} and {\em Desarguesian spread set}.

\section{Spreads of $\PG(rn-1,q)$ containing $r$ normal elements in general position}\label{r normal elements}

We denote the points of $\PG(rn-1,q)$ by $\{(x_1,\ldots,x_r)_{\F_q} \mid x_i\in\F_{q^n}\}$.
Consider a spread set $\mathbf{M}$, containing $0$ and $I$. We define the following $(n-1)$-spread
$$\S_r(\mathbf{M})= \left\{(A_1, A_2 , \ldots, A_r ) \mid  A_i \in \mathbf{M}, \mbox{ every first non-zero matrix } A_k=I\right\}$$
in $\PG(rn-1,q)$. If $\mathbf{M}$ is a nearfield spread set, then one can check that $\S_r(\mathbf{M})$ has $r$ normal elements $S_i,i=1,\ldots,r$, namely
$$S_1= (I, 0,0, \ldots, 0), S_2= (0,I, 0, \ldots, 0),\ldots,  S_r= (0,0,0, \ldots, 0,I).$$
This follows since $\langle S_i,(A_1, A_2 , \ldots, A_r )\rangle$ is partitioned by the elements
$$\{(A_1,\ldots,A_{i-1},B,A_{i+1},\ldots, A_r ) \mid B \in \mathbf{M}\} \cup \{S_i\}.$$
Moreover, in this case, since $\mathbf{M}$ is closed under multiplication, we can simplify notation such that
$$\S_r(\mathbf{M})= \left\{(A_1, A_2 , \ldots, A_r ) \mid  A_i \in \mathbf{M}\right\}.$$
\begin{theorem}\label{NearfieldSpread}
An $(n-1)$-spread $\S$ in $\PG(rn-1,q), r>2$, having $r$ normal elements in general position is projectively equivalent to $\S_r(\mathbf{M})$, for some nearfield spread set $\mathbf{M}$.
\end{theorem}

\begin{proof}
Suppose $r=3$, the points of $\PG(3n-1,q)$ correspond to the set $\{(x,y,z)_{\F_q} \mid x,y,z\in\F_{q^n}, (x,y,z)\neq(0,0,0)\}$.
Consider an $(n-1)$-spread $\S$ of $\PG(3n-1,q)$ having normal elements $S_1,S_2,S_3$ in general position.
Without loss of generality, we may assume that $S_1=(I,0,0), S_2=(0,I,0)$ and $S_3=(0,0,I)$.

As $S_1,S_2$ and $S_3$ are normal elements, the intersection of $\S$ with the $(2n-1)$-spaces $\langle S_1,S_2\rangle$, $\langle S_1,S_3\rangle$ and $\langle S_2,S_3\rangle$ are $(n-1)$-spreads. Hence, there exist spread sets $\mathbf{M_1}, \mathbf{M_2}$ and $\mathbf{M_3}$ (all containing 0 and $I$) such that
\begin{eqnarray*}
\S \cap \langle S_2,S_3\rangle &=& \{P_{A}=(0,A,I) \mid A \in \mathbf{M_1}\}\cup \{ S_2=P_\infty=(0,I,0)\},\\
\S \cap \langle S_1,S_2\rangle &=& \{Q_{B}=(I,B,0) \mid B \in \mathbf{M_2}\}\cup \{ S_2=Q_\infty=P_\infty=(0,I,0)\},\\
\S \cap \langle S_1,S_3\rangle &=& \{R_{C}=(C,0,I) \mid C \in \mathbf{M_3}\}\cup \{ S_1=R_\infty=(I,0,0)\}.
\end{eqnarray*}

As $S_2$ and $S_3$ are normal elements of $\S$, we can obtain every element of $\S$, not on $\langle S_2,S_3\rangle$, as
$$(I,B,C^{-1}) = \langle S_2, R_{C}\rangle \cap \langle S_3, Q_{B}\rangle,$$
with $B\in \mathbf{M_2}$ and $C\in \mathbf{M_3}$.
For any $B\in \mathbf{M_2}$ and $C\in \mathbf{M_3}$, consider the following projections of elements of $\S$ from $S_1$ onto $\langle S_2,S_3\rangle$:
\begin{eqnarray*}
(0,B,I)&=&\langle S_1, (I,B,I)\rangle\cap\langle S_2,S_3\rangle,\\
(0,C,I)&=& \langle S_1,(I,I,C^{-1})\rangle\cap\langle S_2,S_3\rangle,\\
(0,BC,I) &=&\langle S_1, (I,B,C^{-1})\rangle\cap\langle S_2,S_3\rangle.
\end{eqnarray*}
As $S_1$ is a normal element, these subspaces are all contained in $\S$. From the first two, it follows that $\mathbf{M_2}$ and $\mathbf{M_3}$ are contained in $\mathbf{M_1}$, hence $\mathbf{M_1}=\mathbf{M_2}=\mathbf{M_3}$.
Using the third, we find that $BC\in\mathbf{M_1}$, i.e. $\mathbf{M_1}$ is closed under multiplication.
By Theorem \ref{SemiNear}, we conclude that the spread $\S \cap \langle S_2,S_3\rangle$ (and thus also $\S \cap \langle S_1,S_2\rangle$ and $\S \cap \langle S_1,S_3\rangle$) is a nearfield spread.

The result now follows for $r=3$, since we have obtained that
$$\S=\S_3(\mathbf{M_1})= \left\{(A_1, A_2 , A_3 ) \mid  A_i \in \mathbf{M_1}\right\},$$
where $\mathbf{M_1}$ is a nearfield spread set.

By induction, suppose the result is true for $r=t-1\geq 3$. We will now prove it is true for $r=t$.
Consider an $(n-1)$-spread $\S$ of $\PG(tn-1,q)$ having $t$ normal elements $S_1,\ldots,S_t$ in general position. Without loss of generality, we may assume $$S_1= (I, 0,0, \ldots, 0), S_2= (0,I, 0, \ldots, 0),\ldots,  S_t= (0,0,0, \ldots, 0,I).$$

Consider the $((t-1)n-1)$-subspaces $\Pi_1=\langle S_2, S_3,\ldots,S_{t}\rangle$ and $\Pi_2=\langle S_1, S_3,\ldots,S_{t}\rangle$. Clearly, $\Pi_1$ corresponds to the points with coordinates  $\{(0,x_2,\ldots,x_t)_{\F_q} \mid x_i\in\F_{q^n}\}$ and $\Pi_2$ corresponds to the points with coordinates  $\{(x_1,0,x_3,\ldots,x_r)_{\F_q} \mid x_i\in\F_{q^n}\}$.  Since all $S_j$ are normal elements, we have that $\S_{i}=\S\cap\Pi_i$ is an $(n-1)$-spread of $\Pi_i$ containing $t-1$ normal elements in general position. By the induction hypothesis, there exist nearfield spread sets $\mathbf{M_1}$ and $\mathbf{M_2}$ such that
\begin{align*}
\S_1 =&\, \left\{(0, A_2 ,A_3 , \ldots, A_t ) \mid  A_i \in \mathbf{M_1}\right\}, \mbox{ and }\\
\S_2 =&\, \left\{(A_1, 0,A_3 , \ldots, A_t ) \mid  A_i \in \mathbf{M_2}\right\}.
\end{align*}

The spreads $\S_1$ and $\S_2$ overlap in the $((t-2)n-1)$-space $\Pi_1\cap\Pi_2$, hence we find that $\mathbf{M_1}=\mathbf{M_2}$.

All elements of $\S$, not on $\langle S_1,S_2\rangle$, are of the form $\langle S_1, U\rangle\cap\langle S_2,V\rangle$, for some $U\in\S_1,V\in\S_2$. To find the coordinates of the elements of $\S\cap\langle S_1, S_2\rangle$, we can consider the projection of $\S$ from $S_t$ onto $\langle S_1,S_2,\ldots,S_{t-1}\rangle$. We obtain that
$$\S= \left\{(A_1, A_2 , \ldots, A_t ) \mid  A_i \in \mathbf{M_1}\right\}.$$\end{proof}
\section{Characterising Desarguesian spreads}\label{r+1 normal elements}
In this section, we provide two characterisations of Desarguesian spreads, one dependent and one independent of $n$ and $q$. We first need the characterisation of finite nearfields.
\begin{definition}
A pair of positive integers $(q, n)$ is called a {\em Dickson number pair} if it satisfies the following relations:
\begin{enumerate}
\item[(i)] $q = p^h$ for some prime $p$,
\item[(ii)] each prime divisor of $n$ divides $q - 1$,
\item[(iii)] if $q \equiv 3 \mod 4$ , then $n \not\equiv 0 \mod 4$.
\end{enumerate}
\end{definition}

By \cite{EllersKarzel} and \cite{Zassenhaus}, there is a uniform method for constructing a finite nearfield of order $q^n$, with centre and kernel isomorphic to $\F_q$, whenever $(q, n)$
is a Dickson number pair. Such a nearfield is called a {\em Dickson nearfield} or a {\em regular nearfield}.
Moreover, by  \cite{Zassenhaus}, apart from seven exceptions, every finite nearfield, which is not a field, is a Dickson nearfield.
These seven nearfield exceptions have parameters $n=2$ and $q\in \{5,7,11,23,29,59\}$; note that there are two non-equivalent non-regular nearfields for $(q,n)=(11,2)$.

If $\S$ contains $r$ normal elements in general position, then $\S$ is a Desarguesian spread.
\begin{theorem}
Consider an $(n-1)$-spread $\S$ in $\PG(rn-1,q)$, $r>2$, such that for every divisor $k | n$, we have that $(q^k,\frac{n}{k})$ is not a Dickson number pair and does not correspond to the parameters of one of the seven nearfield exceptions.
If $\S$ contains $r$ normal elements in general position, then $\S$ is a Desarguesian spread.
\end{theorem}
\begin{proof}
By Theorem \ref{NearfieldSpread} the spread $\S$ is projectively equivalent to $\S_r(\mathbf{M})$, for a nearfield spread set $\mathbf{M}$.  By assumption, for every divisor $k|n$, a nearfield of order $q^n$, having kernel $\F_{q^k}$, is a field, hence $\mathbf{M}$ is a Desarguesian spread set. 
It follows that $\S$ is a Desarguesian spread.
\end{proof}

\begin{lemma}\label{UniqueDesSpread}{\rm \cite[Lemma 13]{WijEggs}}
Let $\D_0$ be a Desarguesian $(n-1)$-spread in a $((t-1)n-1)$-dimensional subspace $\Pi$ of $\PG(tn-1,q)$, let $T$ be an element of $\D_0$ and let $E_1$ and $E_2$ be mutually disjoint $(n-1)$-spaces such that $\langle E_1,E_2\rangle$ meets $\Pi$ exactly in the space $T$. Then there exists a unique Desarguesian $(n-1)$-spread of $\PG(tn-1,q)$ containing $E_1$, $E_2$ and all elements of $\D_0$.
\end{lemma}

\begin{theorem}
\label{r+1normal}
Consider an $(n-1)$-spread $\S$ in $\PG(rn-1,q)$, $r>2$.
If $\S$ contains $r+1$ normal elements in general position, then $\S$ is a Desarguesian spread.
\end{theorem}
\begin{proof}
Suppose $r=3$, the points of $\PG(3n-1,q)$ correspond to the set $\{(x,y,z)_{\F_q} \mid x,y,z\in\F_{q^n}, (x,y,z)\neq(0,0,0)\}$.
Suppose $\S$ contains normal elements $S_1,S_2,S_3,S_4$ in general position.
Without loss of generality, we may assume that $S_0=(I,I,I)$, $S_1=(I,0,0)$, $S_2=(0,I,0)$ and $S_3=(0,0,I)$.

As $S_1,S_2,S_3$ are normal elements, by following the proof of Theorem \ref{NearfieldSpread}, we see that
$$\S=\S_3(\mathbf{M})=\{(A_1, A_2 ,A_3 ) \mid  A_i \in \mathbf{M}\},$$
for a nearfield spread set $\mathbf{M}$ (containing 0 and $I$).

Given $A\in\mathbf{M}\setminus\{0,I\}$, consider the spread element $R_A=(0,A,I) \in \S \cap \langle S_2,S_3\rangle$, and look at the element
$$\langle S_0, R_A \rangle\cap\langle S_1,S_2\rangle= \left(I,I-A,0\right).$$
As $S_0$ is a normal element for $\S$, this $(n-1)$-space is contained in $\S \cap \langle S_1,S_2\rangle$. It follows that the matrix $I-A \in \mathbf{M}$.

As $\mathbf{M}$ is closed under multiplication, for all $A,B\in\mathbf{M}$, we have $ B-BA\in\mathbf{M}$.
Given matrices $B,C\in\mathbf{M}$, there exists a unique $A\in\mathbf{M}$ for which $BA=C$. Hence, for all $B,C\in\mathbf{M}$, we find that $B-C= B-BA$ is contained in $\mathbf{M}$. It follows that $\mathbf{M}$ is also closed under addition, hence $\mathbf{M}$ is a Desarguesian spread set. We conclude that $\S$ is a Desarguesian spread.

By induction, suppose the result is true for $r=t-1$. We will now prove it is true for $r=t$.
Consider an $(n-1)$-spread $\S$ of $\PG(tn-1,q)$ having $t+1$ normal elements $S_0,\ldots,S_{t}$ in general position.
Consider the $((t-1)n-1)$-subspace $\Pi=\langle S_1,S_2, \ldots,S_{t-1}\rangle$. As all $S_j$ are normal, $\S\cap\Pi$ is an $(n-1)$-spread containing $t-1$ normal elements.

Consider the $(n-1)$-space $T=\langle S_{0},S_{t}\rangle\cap\Pi$, clearly $T\in\S\cap\Pi$.
Take an element $R \in \S\cap\Pi$ different from $T$ and consider the $(3n-1)$-space $\pi=\langle S_0, S_{t},R\rangle$. Note that the intersection $\Pi\cap\pi$ contains the elements $T$ and $R$. Since $S_0, S_{t}$ are normal elements, $\S\cap\pi$ is an $(n-1)$-spread. Both $\S\cap\Pi$ and $\S\cap\pi$ are $(n-1)$-spreads, hence $\S\cap(\Pi \cap \pi)=\S\cap\langle T,R\rangle$ is an $(n-1)$-spread. It follows that $T$ is a normal element for the spread $\S\cap\Pi$.

Since the elements $S_j$ lie in general position with respect to the full space $\PG(tn-1,q)$, the elements $S_1,S_2,\ldots,S_{t-1},T$ lie in general position with respect to $\Pi$. Hence, the spread $\S\cap\Pi$ contains $t$ normal elements in general position, and thus, by the induction hypothesis, $\S\cap\Pi$ is a Desarguesian spread.

By Lemma \ref{UniqueDesSpread}, there exists a unique Desarguesian spread $\D$ of $\PG(tn-1,q)$ containing $S_0,S_t$ and all elements of $\S\cap\Pi$. Every element of $\D$, not on $\langle S_0, S_{t}\rangle$, can be obtained as $\langle S_0, U \rangle \cap \langle S_{t}, V \rangle$, for some $U,V\in\S\cap\Pi$. Since $S_0$ and $S_{t}$ are normal elements of $\S$, all elements of $\D \setminus \langle S_0, S_{t}\rangle$ are elements of $\S$. Looking from the perspective from $S_1$, it is easy to see that the elements of $\D \cap \langle S_0, S_{t}\rangle$ are also elements of $\S$.

It follows that $\S=\D$ is a Desarguesian spread.
\end{proof}

\section{Spreads of $\PG(3n-1,q)$ containing 3 normal elements in a $(2n-1)$-space}\label{3 on a line}

In this section, we will characterise $(n-1)$-spreads of $\PG(3n-1,q)$ containing 3 elements contained in the same $(2n-1)$-space. First, we need to introduce some definitions and notations concerning the (restricted) closure of a point set and the field reduction of sublines and subplanes.

\subsection{The (restricted) closure of a point set}

\begin{definition}
An {\em $\F_{q_0^{}}$-subline} in $\PG(1,q^n)$, for a given subfield $\F_{q_0^{}}\leq\F_q$, is a set of $q_{0^{}}+1$ points in $\PG(1,q^n)$ that is projectively equivalent to the set $\{(x,y)_{\F_{q^n}}\mid(x,y)\in (\F_{q_0^{}})^2 \setminus (0,0)\}$.
\end{definition}
As $\PGL(2,q^n)$ acts $3$-transitively on the points of the projective line, we see that any $3$ points define a unique $\F_{q_0^{}}$-subline.

Consider the field reduction map $\mathcal{F}$ from $\PG(1,q^n)$ to $\PG(2n-1,q)$. When we apply $\mathcal{F}$ to an $\F_{q_0^{}}$-subline of $\PG(1,q^n)$, for a subfield $\F_{q_0^{}}\leq\F_q$, we obtain a set $\R_{q_0^{}}$ consisting of ${q_0^{}}+1$ $(n-1)$-spaces, such that if a line meeting $3$ elements $A,B,C$ of $\R_{q_0^{}}$, in the points $P_A,P_B,P_C$ respectively, then the unique $\F_{q_0^{}}$-subline containing $P_A,P_B,P_C$  meets all elements of $\R_{q_0^{}}$.
When $\F_{q_0^{}}=\F_q$, such a set $\R_q$ is called a {\em regulus}. It is well known that $3$ mutually disjoint $(n-1)$-spaces $A,B,C$ in $\PG(2n-1,q)$ lie on a unique regulus, which we will notate by $\R_q(A,B,C)$.
It now easily follows that $3$ mutually disjoint $(n-1)$-spaces $A,B$ and $C$ in $\PG(2n-1,q)$ lie on a unique $\R_{q_0^{}}$, which we will notate by $\R_{q_0^{}}(A,B,C)$. Note that $\R_{q_0^{}}(A,B,C) \subset \R_{q}(A,B,C)$.

\begin{definition}
An {\em $\F_{q_0^{}}$-subplane} in $\PG(2,q^n)$, for a given subfield $\F_{q_0^{}}\leq\F_q$, is a set of $q_{0^{}}^2+q_{0^{}}+1$ points and $q_{0^{}}^2+q_{0^{}}+1$ lines in $\PG(2, q^n)$ forming a projective plane, where the point set is projectively equivalent to the set $\{(x_0, x_1, x_2)_{\F_{q^n}} \mid (x_0, x_1, x_2) \in (\F_{q_0^{}})^3 \setminus (0, 0, 0)\}$.
\end{definition}
As $\PGL(3, q^n)$ acts sharply transitively on the frames of $\PG(2, q^n)$, we see that 4 points in general position define a unique $\F_{q_0^{}}$-subplane of $\PG(2, q^n)$.

Consider the field reduction map $\mathcal{F}$ from $\PG(2,q^n)$ to $\PG(3n-1,q)$. If we apply $\mathcal{F}$ to the point set of an $\F_{q_0^{}}$-subplane, we find a set $\mathcal{V}_{q_{0^{}}}$ of $q_{0^{}}^2+q_{0^{}}+1$ elements of a Desarguesian spread $\D$.
When $\F_{q_0^{}}=\F_q$, the set $\mathcal{V}_{q}$ consists of one system of a Segre variety $\mathbf{S}_{n-1,2}$ (for more information see e.g. \cite{MichelGeertruiFieldReduction}).
By \cite[Proposition 2.1, Corollary 2.3, Proposition 2.4]{LavrauwZanella}, we know that four $(n -1)$-spaces $A,B,C,D$ in $\PG(3n -1, q)$ in general position are contained in a unique Segre variety $\mathcal{V}_{q}$, which we will notate by $\mathcal{V}_{q}(A,B,C,D)$.
As a corollary, we obtain that for any subfield $\F_{q_0^{}}\leq\F_q$, four $(n-1)$-spaces $A,B,C,D$ in $\PG(3n -1, q)$ in general position are contained in a unique $\mathcal{V}_{q_{0^{}}}$, which we will notate by $\mathcal{V}_{q_{0^{}}}(A,B,C,D)$. Note that for any three $(n-1)$-spaces $E_1,E_2,E_3$ in $\mathcal{V}_{q_{0^{}}}(A,B,C,D)$, contained in the same $(2n-1)$-space, we have that $\R_{q_0}(E_1,E_2,E_3) \subset \mathcal{V}_{q_{0^{}}}(A,B,C,D)$.

\begin{definition}\label{closure}
If a point set $S$ contains a frame of $\PG(2,q)$, then its {\em closure} $\overline{S}$ consists of the points of the
smallest $2$-dimensional subgeometry of $\PG(2,q)$ containing all points of $S$.
\end{definition}
The closure $\overline{S}$ of a point set $S$ can be constructed recursively as follows:
\begin{itemize}
\item[(i)] determine the set $\A$ of all lines of $\PG(2,q)$ spanned by two points of $S$;
\item[(ii)] determine the set $\overline{S}$ of points that occur as the exact intersection of
two lines in $\A$, if $\overline{S} \neq S$ replace $S$ by
$\overline{S}$ and go to~(i), otherwise stop.
\end{itemize}

This recursive construction coincides with the definition of the closure of a set of points in a plane containing a quadrangle, given in \cite[Chapter XI]{Hughes}.

Consider a set $S$ of points of $\PG(2,q)$ containing a specific subset $\{P_i\}_{i=1,\ldots,m}$.
We define the {\em restricted closure} $\widetilde{S}$ {\em with respect to the points} $\{P_i\}$ to be the point set constructed recursively as follows:
\begin{itemize}
\item[(i)] determine the set $\A$ of all lines of $\PG(2,q)$ of the form $\langle P_i, Q \rangle$, $i=1,\ldots,m$, $Q \in S$;
\item[(ii)] determine the set $\widetilde{S}$ of points that occur as the exact intersection of
two lines in $\A$, if $\widetilde{S} \neq S$ replace $S$ by
$\widetilde{S}$ and go to~(i), otherwise stop.
\end{itemize}
Clearly the restricted closure $\widetilde{S}$ of $S$, with respect to all its points, is exactly its closure $\overline{S}$.

\begin{lemma}\label{subplane} Consider the point set $S$ of $\PG(2,q)$ containing a frame $\{P_1,P_2,Q_1,Q_2\}$ and the point $P_3 =P_1P_2 \cap Q_1Q_2$. The points of the restricted closure $\widetilde{S}$ of $S$ with respect to $\{P_1,P_2,P_3\}$, not on the line $P_1P_2$, are the points of the affine $\F_p$-subplane $\overline{S}\,\backslash P_1P_2$.
\end{lemma}
\begin{proof}
Clearly, the set $\widetilde{S}$ can contain no other points than the points of the $\F_p$-subplane $\overline{S}$.

Without loss of generality, we may choose coordinates such that $P_1=(1,0,0)_{\F_q}$, $P_2=(0,0,1)_{\F_q}$, $Q_1=(0,1,0)_{\F_q}$, $Q_2=(1,1,1)_{\F_q}$ and hence $P_3=(1,0,1)_{\F_q}$.

The point $U_1=(0,1,1)_{\F_q} =P_1 Q_2 \cap  P_2Q_1$ is contained in $\widetilde{S}$.
Hence, the point $T_2=(1,1,2)_{\F_q}=P_3 U_1 \cap P_2Q_2 $ is contained in $\widetilde{S}$.
Using $T_2\in\widetilde{S}$, we see that the point $U_2=(0,1,2)_{\F_q}=P_1 T_2 \cap P_2Q_1$ is also contained in $\widetilde{S}$.

Continuing this process, we get that all points $U_a=(0,1,a)_{\F_q}\in P_2Q_1$ and $T_a=(1,1,a)_{\F_q}\in P_2 Q_1$, with $a\in \F_p$, are contained in $\widetilde{S}$. All other points of the unique $\F_p$-subplane $\overline{S}$ through $P_1, P_2, Q_1,Q_2$ and not on $P_1P_2$, can be written as the intersection point $\langle P_1, T_a\rangle\cap\langle P_2, U_b\rangle$, for some $a$ and $b$ in $\F_p$.
Hence, they are all contained in $\widetilde{S}$.
\end{proof}

We can translate the previous lemma to the higher dimensional case in the following way.

\begin{lemma}\label{subplanefieldreduction}
Consider an $(n-1)$-spread $\S$ in $\PG(3n-1,q)$, $q=p^h$, having 3 normal elements $S_1, S_2, S_3$ contained in the same $(2n-1)$-space $\Pi_0$. If two spread elements $R_1,R_2 \in \S$ satisfy $\langle R_1,R_2 \rangle\cap\langle S_1,S_2\rangle=S_3$, then all $(n-1)$-spaces of $\mathcal{V}_p(S_1,S_2,R_1,R_2) \setminus \langle S_1,S_2\rangle$ are contained in $\S$.
\end{lemma}
\begin{proof}
There exists a field reduction map $\mathcal{F}$ from $\PG(2,q^n)$ to $\PG(3n-1,q)$, such that $S_1$, $S_2$, $S_3$, $R_1$, $R_2$ are contained in the Desarguesian spread $\D$ of $\PG(3n-1,q)$ defined by $\mathcal{F}$. This means, there exist points $P_1$, $P_2$, $P_3$, $Q_1$, $Q_2$ such that their field reduction is equal to $S_1$, $S_2$, $S_3$, $R_1$, $R_2$ respectively.

Since $S_1,S_2,S_3$ are normal elements for $\S$, the field reduction of the points of the restricted closure of $\{P_1,P_2,P_3,Q_1,Q_2\}$ with respect to $\{P_1,P_2,P_3\}$ must all be contained in $\S$. By  Lemma \ref{subplane}, the restricted closure contains all points of the $\F_p$-subplane defined by $\{P_1,P_2,Q_1,Q_2\}$, not on $\langle P_1, P_2\rangle$.
Hence, its field reduction, i.e. the $(n-1)$-spaces of $\mathcal{V}_p(S_1,S_2,R_1,R_2) \setminus \langle S_1, S_2\rangle$ are all contained in $\S$.
\end{proof}

\subsection{Characterising spreads with 3 normal elements in the same $(2n-1)$-space}


There exists a different but equivalent definition of a semifield spread than the one given in Subsection \ref{quasifieldsection}. Namely, an $(n-1)$-spread $\S$ of $\PG(2n-1,q)$ is a semifield spread if and only if it contains a special element $E$ such that the stabiliser of $\S$ fixes $E$ pointwise and acts transitively on the elements of $\S \setminus\{E\}$ (see \cite[Corollary 5.60]{HandbookTranslationPlanes}). The element $E$ is called the {\em shears} element \index{shears element}
of $\S$.

Consider a matrix spread set $\Mb$ containing $0$ and closed under addition, and consider the corresponding semifield spread $\S(\Mb)= \{(I,A)\mid A \in \mathbf{M}\} \cup \{E=(0,I)\}$ of $\PG(2n-1,q)$. In this case, the element $E=(0,I)$ is the shears element of $\S(\Mb)$.


\begin{definition}
The subsets of a semifield $\mathcal{Q}$ given as
\begin{align*}
\NN_r(\mathcal{Q}) &= \{x \in \mathcal{Q} \mid \forall a, b \in  {Q}: (a * b) * x = a * (b * x)\}\\
\NN_m(\mathcal{Q}) &= \{x \in \mathcal{Q} \mid\forall a, c \in {Q}: (a * x) * c = a * (x * c)\},\\
Z(\mathcal{Q}) &= \{x \in \NN_r(\mathcal{Q}) \cap \NN_m(\mathcal{Q}) \mid \forall a \in \mathcal{Q} : x * a = a * x \}
\end{align*}
are all fields and are called, respectively, the {\em right nucleus}, {\em middle nucleus} 
and {\em centre} of the semifield.
\end{definition}\index{right nucleus}\index{middle nucleus} \index{centre} 


The parameters of the semifield $\mathcal{Q}_e(\Mb)$ can be translated to subsets of the associated spread set $\Mb$, that is
$\mathcal{N}_r(\mathbf{M})=\{ M_x \in \mathbf{M} \mid x \in \NN_r(\mathcal{Q}_e(\mathbf{M}))\}$, $\mathcal{N}_m(\mathbf{M})=\{ M_x \in \mathbf{M}\mid x \in \NN_m(\mathcal{Q}_e(\mathbf{M}))\}$ and $\mathcal{Z}(\mathbf{M})=\{ M_x \in \mathbf{M} \mid x \in Z(\mathcal{Q}_e(\mathbf{M}))\}$.
As in \cite[Theorem 2.1]{EbertMarinoPolverino}, but considering our conventions, we obtain that
\begin{align*}
\NN_r(\mathcal{Q}_e(\Mb)) &= \{x \in \mathcal{Q}_e(\Mb) \mid \forall a, b \in  \mathcal{Q}_e(\Mb): (a * b) * x = a * (b * x)\}\\
&=\{x \in \mathcal{Q}_e(\Mb) \mid \forall a, b \in  \mathcal{Q}_e(\Mb): a M_b M_x = a M_{b * x} \}\\
&=\{x \in \mathcal{Q}_e(\Mb) \mid \forall  b \in  \mathcal{Q}_e(\Mb): M_b M_x = M_{b * x} \}\\
&=\{x \in \mathcal{Q}_e(\Mb) \mid \forall b \in  \mathcal{Q}_e(\Mb): M_b M_x \in \Mb \}\\
&=\{x \in \mathcal{Q}_e(\Mb) \mid \forall M \in \Mb: M M_x \in \Mb \}.
\end{align*}
Hence, it follows that
$$
\mathcal{N}_r(\mathbf{M}) =  \{X \in \Mb \mid \mathbf{M}X = \mathbf{M} \mbox{ or } X=0\}.
$$
Similarly, we obtain \begin{align*}
\mathcal{N}_m(\mathbf{M}) = \,& \{X \in \Mb \mid X\mathbf{M} = \mathbf{M} \mbox{ or  } X=0\},  \mbox{ and }\\
\mathcal{Z}(\mathbf{M}) = \, & \{X \in \mathcal{N}_r(\Mb)\cap\mathcal{N}_m(\Mb) \mid\forall Y \in \Mb: YX=XY \}.
\end{align*}

Suppose $\Mb$ contains the identity $I$, then $\Mb$ is a subspace over a subfield $\F_{q_0}\leq \Fq$  if and only if $\{\lambda I \mid \lambda \in \F_{q_0}\}\subseteq \mathcal{Z}(\Mb)$. 

Recall that a spread $\S$ in $\PG(2n-1,q)$ is {\em regular} if and only if for any $3$ elements $E_1,E_2,E_3\in\S$, the elements of $\R_q(E_1,E_2,E_3)$ are all contained in $\S$.
It is well known that every $(n-1)$-spread of $\PG(2n-1,q)$, $q>2$, is regular if and only if it is Desarguesian (see \cite{BrB1966}). When $q=2$, every $(n-1)$-spread  of $\PG(2n-1,2)$ is regular. More generally, for every three elements $E_1,E_2,E_3$ of an $(n-1)$-spread $\S$ of $\PG(2n-1,2^h)$, all elements of $\R_2(E_1,E_2,E_3)=\{E_1,E_2,E_3\}$ are contained in $\S$.

Loosening the concept of regularity of Desarguesian spreads, we obtain the following result for semifield spreads. For $\F_{q_0}=\F_q$, this result was already obtained in \cite[Teorema 5]{Lunardon Proposizioni}.


\begin{theorem}\label{IsSemifieldSpread}
Suppose $\S$ is an $(n-1)$-spread of $\PG(2n-1,q)$. Consider a subfield $\F_{q_0^{}}\leq\F_q$, $q_0>2$ and an element $E \in\S$.
The spread $\S$ is a semifield spread with $\F_{q_0}$ contained in its center and having $E$ as its shears element if and only if for all $E_1,E_2\in\S$, we have $\R_{q_0^{}}(E,E_1,E_2)\subset\S$.
\end{theorem}
\begin{proof}
We may assume without loss of generality that $\S =\S(\Mb)= \{(I,A)\mid A \in \mathbf{M}\} \cup \{E=(0,I)\}$ for some spread set $\mathbf{M}$ containing $0$ and $I$. Let $E_i = (I,A_i)$. Then the set $\R_{q_0}(E,E_1,E_2)$ consists of the spaces $(I,(1-\lambda)A_1+\lambda A_2)$ for $\lambda\in \F_{q_0}$, together with $E$. Hence $(1-\lambda)A_1+\lambda A_2\in \Mb$, for all $A_1,A_2\in \Mb$ and all $\lambda \in \F_{q_0}$.

Since $0\in\Mb$, for all $A_2\in \Mb$ we get that $\lambda A_2\in \Mb$ for all $\lambda \in \F_{q_0}$.
Therefore, for all $A_1,A_2\in \Mb$,
$\mu A_1+\lambda A_2\in \Mb$ for all $\lambda,\mu \in \F_{q_0}$, and so $\Mb$ is an $\F_{q_0}$-subspace, implying that $\Mb$ is a semifield spread set with centre containing $\F_{q_0}$. It follows that $\S$ is a semifield spread with shears element $E$.
\end{proof}

\begin{corollary} {\rm \cite[Teorema 5]{Lunardon Proposizioni}}
Suppose $\S$ is an $(n-1)$-spread of $\PG(2n-1,q)$, $q>2$, and consider an element $E \in\S$.
For all $E_1,E_2\in\S$, we have $\R_{q}(E,E_1,E_2)\subset\S$, if and only if $\S$ is a semifield spread with $\F_q$ contained in the center and having $E$ as its shears element.
\end{corollary}

Now, consider three disjoint $(n-1)$-spaces $S_1= (I, 0,0), S_2= (I,I, 0), S_3=(0,I,0)$ of $\PG(3n-1,q)$ contained in the same $(2n-1)$-space $\Pi_0=\{(x,y,0)_{\F_q} \mid x,y \in \F_{q^n}\}$.
Consider two spread sets $\mathbf{M}$ and $\mathbf{M_0}$, both containing $0$ and $I$, and define the following $(n-1)$-spread
$$\T_3(\mathbf{M},\mathbf{M_0})= \left\{(A, B , I) \mid  A,B \in \mathbf{M}\right\}\cup \left\{(I, C , 0)\mid  C \in \mathbf{M_0}\right\}\cup\{(0,I,0)\}$$
in $\PG(3n-1,q)$.
If $\mathbf{M}$ is a semifield spread set, then one can check that $\T_3(\mathbf{M},\Mb_0)$ has at least $3$ normal elements, namely $S_1,S_2$ and $S_3$.

\begin{theorem}\label{SpreadBySemifield}
Consider an $(n-1)$-spread $\S$ in $\PG(3n-1,q)$, $q$ odd.
If $\S$ contains $3$ normal elements contained in the same $(2n-1)$-space $\Pi_0$, then $\S$ is equivalent to $\T_3(\mathbf{M},\mathbf{M_0})$, for some spread set $\mathbf{M_0}$ and a semifield spread set $\mathbf{M}$.

Furthermore, the set of normal elements of $\S$ contained in $\S\cap \Pi_0$ is equivalent to $\{(I,C,0)\mid C \in \Mb_0 \cap \mathcal{N}_r(\Mb)\}\cup\{(0,I,0)\}$.
\end{theorem}
\begin{proof}
Consider an $(n-1)$-spread $\S$ containing normal elements $S_1,S_2,S_3$ contained in the same $(2n-1)$-space $\Pi_0$.
Since $S_3$ is a normal element with respect to $\S$, we can consider an $(2n-1)$-space $\Pi$, meeting $\Pi_0$ in the space $S_3$, such that $\S\cap\Pi$ is an $(n-1)$-spread.
Without loss of generality, we may assume that $S_1= (I, 0,0), S_2= (I,I, 0), S_3=(0,I,0)$, $\Pi_0=\{(x,y,0)_{\F_q} \mid x,y \in \F_{q^n}\}$ and $\Pi=\{(0,y,z)_{\F_q} \mid y,z \in \F_{q^n}\}$.

Since $\S\cap\Pi$ is an $(n-1)$-spread, there exists a spread set $\mathbf{M}$ such that
$$\S\cap\Pi=\left\{(0, A,I) \mid  A \in \mathbf{M}\right\}\cup\{(0,I,0)\}.$$

Consider two elements $R_1,R_2$ of $\S\cap\Pi$ different from $S_3$. By Lemma \ref{subplanefieldreduction}, all elements of $\mathcal{V}_p(S_1,S_2,R_1,R_2)\setminus\Pi_0$ are contained in $\S$.
It follows that the $p+1$ elements of $\R_p(S_3,R_1,R_2)$ are all contained in $\S\cap\Pi$. By Corollary \ref{IsSemifieldSpread}, we see that $\S\cap\Pi$ is a semifield spread (with $\F_p$ contained in its center and $S_3$ its shears element) hence $\mathbf{M}$ is a semifield spread set.

The elements $S_1$ and $S_2$ are normal elements, hence every element of $\S\setminus \langle S_1,S_2\rangle$ is of the form
$$(B-D,B,I)=\langle S_1, (0,B,I)\rangle\cap\langle S_2,(0,D,I)\rangle,$$
for some $B,D\in\mathbf{M}$.
Since $\mathbf{M}$ is closed under addition, every element of $\S\setminus \langle S_1,S_2\rangle$ is of the form $(A,B,I)$, with $A,B\in\mathbf{M}$.
We conclude that
$$\S= \left\{(A, B , I) \mid  A,B \in \mathbf{M}\right\}\cup \left\{(I, C , 0) \mid  C \in \mathbf{M_0}\right\}\cup\{(0,I,0)\},$$
for some spread set $\mathbf{M_0}$ (containing $0$ and $I$) and a semifield spread set $\mathbf{M}$.

Now suppose $(I,C,0) \in \Pi_0$, $C\in\Mb_0$, is a normal element for $\T_3(\mathbf{M},\mathbf{M_0})$. Given elements $A,B,D,E \in \Mb$, the spread element $(D,E,I)$ is contained in $\langle(I,C,0),(A,B,I)\rangle$ if and only if $(D-A)C = E-B$.
Hence $\langle(I,C,0),(A,B,I)\rangle$ is partitioned by elements of $\T_3(\mathbf{M},\mathbf{M_0})$ if and only if $(D-A)C+B \in \Mb$ for all $A,B,D \in \Mb$. As $\Mb$ is closed under addition, this occurs if and only if $\Mb C \subseteq \Mb$, i.e.\ $C \in \Mb_0 \cap \mathcal{N}_r(\Mb)$.
\end{proof}

\begin{corollary}\label{4elementsDesarguesian}
Consider an $(n-1)$-spread $\S$ in $\PG(3n-1,q)$, $q$ odd.
If $\S$ contains $4$ normal elements not all contained in the same $(2n-1)$-space, then $\S$ is Desarguesian.
\end{corollary}

\begin{proof}
It is clear that if a spread contains three normal elements, either these elements are contained in the same $(2n-1)$-space or they span a $(3n-1)$-space.

If the 4 normal elements lie in general position, the result follows from Theorem \ref{UniqueDesSpread}.

If the 4 normal elements do not lie in general position, we have that 3 of the elements do lie in general position, but the fourth is contained in a $(2n-1)$-space spanned by two of the other normal elements.
Without loss of generality, we may assume that the 4 normal elements have coordinates $(I, 0,0), (0,I, 0), (0,0,I)$ and $(I,I,0)$. From Theorem \ref{NearfieldSpread} we know there exists 
a nearfield spread set $\mathbf{M}$ such that
$$\S=\S_3(\mathbf{M})= \left\{(A_1, A_2 , A_3) \mid  A_i \in \mathbf{M}\right\}.$$
However, from Theorem \ref{SpreadBySemifield} it follows that there exists some spread set $\mathbf{M}_0$ and some semifield spread set $\mathbf{M}'$ such that
$$\S=\T_3(\mathbf{M}',\mathbf{M_0})= \left\{(A, B , I) \mid  A,B \in \mathbf{M}'\right\}\cup \left\{(I, C , 0)\mid  C \in \mathbf{M_0}\right\}\cup\{(0,I,0)\}.$$

We can conclude that $\mathbf{M}=\mathbf{M}_0=\mathbf{M}'$ is both a nearfield spread set as a semifield spread set, hence is a Desarguesian spread set. It now follows that $\S$ is a Desarguesian spread.
\end{proof}

\begin{theorem}\label{r+1normalnotgeneral}
Consider an $(n-1)$-spread $\S$ in $\PG(rn-1,q)$, $q$ odd.
If $\S$ contains at least $r+1$ normal elements, such that $r$ of the elements span the full space, then $\S$ is Desarguesian.
\end{theorem}
\begin{proof}

The spread $\S$ contains $r+1$ normal elements $S_i, i=1,\ldots,r+1$.
By assumption, $r$ of the elements span the full space, that is, lie in general position. Without loss of generality we can choose coordinates such that $$S_1= (I, 0,0, \ldots, 0), S_2= (0,I, 0, \ldots, 0),\ldots,  S_r= (0,0,0, \ldots, 0,I).$$

From Theorem \ref{NearfieldSpread} it follows that there exists a nearfield spread set $\mathbf{M}$, containing $0$ and $I$, such that
$$\S=\S_r(\mathbf{M})=\{(A_1,\ldots,A_r) \mid A_i \in \mathbf{M}\}.$$

Consider the smallest number $t\in\mathbb{N}$ such that $S_{r+1}$ is contained in a $(tn-1)$-subspace spanned by $t$ other normal elements. We may assume that $$S_{r+1}\in \Pi=\langle S_1, S_2,\ldots,S_t\rangle.$$

Suppose $t=2$, that is, $S_{r+1} \in \langle S_1,S_2\rangle$. Clearly, the $(3n-1)$-space $\langle S_1,S_2,S_3\rangle$ contains 4 normal elements, not all in the same $(2n-1)$-space. Hence, from Corollary \ref{4elementsDesarguesian} it follows that $\S \cap \langle S_1,S_2,S_3\rangle$ is Desarguesian. However, since $\S=\S_r(\mathbf{M})$ is completely defined by the spread set $\mathbf{M}$, this spread set is a Desarguesian spread set. 

Suppose $t>2$, that is, the elements $S_1,S_2,\ldots,S_t,S_{r+1}$ lie in general position with respect to the space $\Pi$. 
From Theorem \ref{UniqueDesSpread} it follows that the spread $\S \cap \Pi$ is Desarguesian. Hence, since $\S=\S_r(\mathbf{M})$ is completely defined by the spread set $\mathbf{M}$, this spread set is a Desarguesian spread set. 

We conclude that $\S$ is Desarguesian.
\end{proof}

\section{Spreads of $\PG(rn-1,q)$ containing less than $r$ normal elements}\label{lessthanr}

In this section we consider what sort of spreads arise in $\PG(rn-1,q)$ with less than $r$ normal elements.


We denote the points of $\PG(rn-1,q)$ by $\{(x_1,\ldots,x_r)_{\F_q} \mid x_i\in\F_{q^n}\}$.
Consider $r-1$ spread sets $\mathbf{M}_1,\ldots,\mathbf{M}_{r-1}$, containing $0$, and a nearfield spread set $\mathbf{M}$, containing $0$ and $I$. 
We define the following $(n-1)$-spread:
$$\mathcal{U}_r(\mathbf{M},\mathbf{M}_1,\ldots,\mathbf{M}_{r-1})= 
\left\{(0,A_1, A_2 , \ldots, A_{r-1} ) \mid  A_i \in \mathbf{M}\right\}
\cup 
\left\{(I,B_1, B_2 , \ldots, B_{r-1} ) \mid B_i \in \mathbf{M}_i\right\}$$
in $\PG(rn-1,q)$.
One can check that $\mathcal{U}_r(\mathbf{M},\mathbf{M}_1,\ldots,\mathbf{M}_{r-1})$ contains $r-1$ normal elements $S_i,i=1,\ldots,r-1$, namely
$$S_1= (0,I, 0,0, \ldots, 0), S_2= (0,0,I, 0, \ldots, 0),\ldots,  S_{r-1}= (0,0,0, \ldots, 0,I).$$

\begin{theorem}\label{6.1}
Consider an $(n-1)$-spread $\S$ in $\PG(rn-1,q), r>2$. If $\S$ has $r-1$ normal elements that span a $((r-1)n-1)$-space, then $\S$ is projectively equivalent to $\mathcal{U}_r(\mathbf{M},\mathbf{M}_1,\ldots,\mathbf{M}_{r-1})$, for some spread sets $\mathbf{M}_1,\ldots,\mathbf{M}_{r-1}$, containing $0$, and a nearfield spread set $\mathbf{M}$, containing $0$ and $I$.
\end{theorem}
\begin{proof}
Consider an $(n-1)$-spread $\S$ of $\PG(rn-1,q)$ having $r-1$ normal elements $S_1,\ldots,S_{r-1}$ that span a $((r-1)n-1)$-space. Without loss of generality, we may assume $$S_1= (0,I, 0,0, \ldots, 0), S_2= (0,0,I, 0, \ldots, 0),\ldots,  S_{r-1}= (0,0,0, \ldots, 0,I).$$

Consider the $((r-1)n-1)$-subspace $\Pi_0=\langle S_1, S_2,\ldots,S_{r-1}\rangle$. Clearly, $\Pi_0$ corresponds to set of points with coordinates  $\{(0,x_2,\ldots,x_{r})_{\F_q} \mid x_i\in\F_{q^n}\}$.  Since $\Pi_0$ contains $r-1$ normal elements in general position, from (the proof of) Theorem \ref{NearfieldSpread} it follows that there exists a nearfield spread set $\mathbf{M}$, containing $0$ and $I$, such that $$\S_0=\S\cap\Pi_0= \left\{(0,A_1, A_2 , \ldots, A_{r-1} ) \mid  A_i \in \mathbf{M}\right\}.$$

All elements of $\S\setminus\S_0$ are disjoint from $\Pi_0$, thus we may assume that $P=(I,0,\ldots,0)$ is contained in $\S$.
As all elements $S_i$ are normal, each $(2n-1)$-space $\langle P,S_i\rangle$ intersects $\S$ in an $(n-1)$-spread. Hence, there exist spread sets $\mathbf{M}_i, i=1,\ldots,r-1$, containing the zero matrix 0, such that 
\begin{align*}
\langle P, S_1 \rangle \cap \S =&\, \left\{(I, B_1 ,0 , \ldots, 0 ) \mid B_1 \in \mathbf{M_1}\right\}\cup\{(0,I,0,\ldots,0)\},\\
\langle P, S_2 \rangle \cap \S  =&\, \left\{(I,0,B_2, 0 , \ldots, 0 ) \mid  B_2 \in \mathbf{M_2}\right\}\cup\{(0,0,I,0\ldots,0)\},\\
\ldots\\
\langle P, S_{r-1} \rangle \cap \S  =&\, \left\{(I,0, \ldots, 0,B_{r-1} ) \mid  B_{r-1} \in \mathbf{M_{r-1}}\right\}\cup\{(0,\ldots,0,I)\}
\end{align*}

Consider the $(3n-1)$-space $\pi=\langle P,S_1,S_2\rangle$. Since $S_1$ and $S_2$ are normal elements, we have that $\pi$ intersects $\S$ in an $(n-1)$-spread. Any element $R \in \S\cap(\pi \setminus \langle S_1,S_2\rangle)$ can be obtained as 
$$R=(I,B_1,B_2,0,\ldots,0)=\langle S_2, (I,B_1,0,\ldots,0)\rangle \cap \langle S_1, (I,0,B_2,0,\ldots,0) \rangle,$$
for some $B_1 \in \mathbf{M}_1, B_2 \in \mathbf{M}_2$.

Similarly, any element in $\S\cap(\langle P,S_1,S_3\rangle \setminus \langle S_1,S_3\rangle)$ can be obtained as 
$$(I,B_1,0,B_3,0,\ldots,0)=\langle S_3, (I,B_1,0,\ldots,0)\rangle \cap \langle S_1, (I,0,0,B_3,0,\ldots,0) \rangle,$$
for some $B_1 \in \mathbf{M}_1, B_3 \in \mathbf{M}_3$.

Consider the $(4n-1)$-space $\pi'=\langle P,S_1,S_2,S_3\rangle$. Since $S_1,S_2, S_3$ are normal elements, we have that $\pi'$ intersects $\S$ in an $(n-1)$-spread. Any element $R' \in \S\cap(\pi' \setminus \langle S_1,S_2,S_3\rangle)$ can be obtained as 
$$R'=(I,B_1,B_2,B_3,0,\ldots,0)=\langle S_3, (I,B_1,B_2,0,\ldots,0)\rangle \cap \langle S_2, (I,B_1,0,B_3,0,\ldots,0) \rangle,$$
for some $B_1 \in \mathbf{M}_1, B_2 \in \mathbf{M}_2, B_3 \in \mathbf{M}_3$.

Continuing this process, we obtain coordinates for all elements of $\S \setminus \S_0$, that is, all elements are of the form 
$$(I,B_1,\ldots,B_{r-1}),$$
for some $B_i \in \mathbf{M}_i, i=1,\ldots, r-1$.
The statement now follows.
\end{proof}

Following the proof of Theorem \ref{6.1} we obtain the following corollary.

\begin{corollary}
Consider an $(n-1)$-spread $\S$ in $\PG(rn-1,q), r>2$. If $\S$ has $k< r-1$ normal elements that span a $(kn-1)$-space $\pi$, then there exist $((k+1)n-1)$-spaces $\Pi_j, j=1,2,\ldots, m=\frac{q^{(r-k-1)n}-1}{q^n-1}$ such that
\begin{itemize}
\item $\forall j \neq j': \Pi_j \cap \Pi_j' = \pi$,
\item $\langle \Pi_1,\ldots,\Pi_m\rangle =\PG(rn-1,q)$,
\item $\pi \cap \S \cong \mathcal{S}_k(\mathbf{M})$,
\item $\Pi_j \cap \S \cong \mathcal{U}_{k+1}(\mathbf{M},\mathbf{M}_{1,j},\ldots,\mathbf{M}_{k,j})$,
\end{itemize}
for some spread sets $\mathbf{M}_{1,1},\ldots,\mathbf{M}_{k,1},\ldots,\mathbf{M}_{1,m},\ldots,\mathbf{M}_{k,m}$ and nearfield spread set $\mathbf{M}$.
\end{corollary}

\begin{remark}
It is clear from the above that it is easy to find spreads in $\PG(rn-1,q)$ with fewer than $r$ normal elements; we simply choose spread sets $\mathbf{M}_i$   which are not nearfield spread sets, and define a spread $\left\{(I, B_2 , \ldots, B_r ) \mid B_i \in \mathbf{M}_i\right\}\cup \S'$, where $\S'$ is some spread of the subspace $\{(0,x_2,\ldots,x_r)_{\F_q} \mid x_i\in\F_{q^n}\}$. The relative abundance of spread sets compared to nearfield spread sets ensures that in general, the majority of spreads will have few normal elements. However, also note that not every spread is of this form, as it is not necessarily true that a spread of $\PG(rn-1,q)$ defines a spread in one (or more) of the $((r-1)n-1)$-dimensional subspaces of $\PG(rn-1,q)$.
\end{remark}

\section{Spreads and translation Sperner spaces}
\label{sperner}

In this section we outline the connection between spreads and incidence structures known as {\it Sperner spaces}. We follow the terminology of \cite{HandbookTranslationPlanes}. In order to translate the results of this paper into results on Sperner spaces, we will need to introduce the concepts of normality and general position in this context.

\subsection{Preliminary definitions}

\begin{definition}
A $2-(v,k,1)$ design is a point-line incidence structure $(\P,\L)$ such that the set $\P$ of {\it points} satisfies $|\P|=v$, the set $\L$ of lines consists of $k$-subsets of $\P$ and any two points are contained in precisely one line. A $2-(v,k,1)$ design is said to have a {\it parallelism} if there is a partition of $\L$ into sets, such that each set defines a partition of $\P$. A $2-(v,k,1)$ design with parallelism is known as a {\it Sperner space} with parameters $(v,k)$. These are sometimes also referred to as {\em weak affine spaces}. A $2-(k^2,k,1)$ design is simply an affine plane.
\end{definition}

By the Barlotti-Cofman construction \cite{BarlottiCofman} every $(n-1)$-spread $\S$ in $\PG(rn-1,q)$ defines a Sperner space $\T(\S)$ with parameters $(q^{rn},q^n)$.
For this construction, we embed $\PG(rn-1,q)$ as a hyperplane $H_\infty$ in $\PG(rn,q)$. We define the design $\T(\S)=(\P,\L)$ whose:
\begin{itemize}
\item
points $\P$ are the {\em affine} points, i.e.\ the points of $\PG(rn,q)\backslash H_\infty$;
\item
lines $\L$ are the $n$-spaces intersecting $H_\infty$ precisely in an element of $\S$;
\item
parallel classes are the sets of $n$-spaces through a given element of $\S$.
\end{itemize}
This construction gives a special type of Sperner space, called a {\it translation Sperner space}. Any translation Sperner space is of this form; see \cite[Chapter 2]{Dembowski} for more on this correspondence. When $r=2$, this is exactly the Andr\'e/Bruck-Bose correspondence between $(n-1)$-spreads of $\PG(2n-1,q)$ and affine translation planes.

Note that when $\S$ is Desarguesian, the design $\T(\S)$ is isomorphic to $\AG(r,q^n)$.

\begin{definition}
Consider a Sperner space $(\P,\L)$. For a subset $\P'$ of $\P$, the {\it linear manifold} $\M(\P')$ of $\P'$, is defined as the induced incidence structure on the smallest subset $\M$ of $\P$ containing $\P'$ such that for any $A,B\in \M$, the point set $\M$ contains all points of the line $\langle A,B\rangle$. For $\P'$ a set of three non-collinear points, the linear manifold defined by $\P'$ is called a {\it pseudo-plane} \cite{BarlottiCofman}. 
For a subset $\L'\subset\L$ of lines, we define $\M(\L')$ to be the linear manifold defined by the union of all points of the lines of $\L'$.
\end{definition}
Note that it is easy to verify that a Sperner space is an affine space if and only if
each of its pseudo-planes is an affine plane.
We wish to translate our results in this context; for this we introduce the following definition.
\begin{definition}
A line $\ell$ in a Sperner space $\T$ will be called {\em normal} if 
every pseudo-plane containing $\ell$ is an affine plane.

\end{definition}

Note that a line is normal if and only if every line in its parallel class is normal, and so we could equally define the normality of a parallel class. A normal line thus implies that $\T$ contains many affine planes. 

We now consider the exact correspondence between a normal element of a spread and a normal line of the corresponding translation Sperner space.
\begin{theorem}
\label{thm:spernernormal}
A line in the translation Sperner space $\T(\S)$ defined by a spread $\S$ is normal if and only if it corresponds to an $n$-space intersecting the hyperplane $H_\infty$ in a normal element of $\S$.
\end{theorem}

\begin{proof}
Consider a spread $\S$ in $\PG(rn-1,q)$ and its corresponding translation Sperner space $\T(\S)=(\P,\L)$ with parameters $(q^n,q^{rn})$.

First, consider a normal element $S$ of $\S$. 
Let $\ell$ be a line of $\L$ corresponding to an $n$-space containing $S$. 
Take any affine point $Q\in\P$ not incident with $\ell$. The line $\ell$ and the point $Q$ define a unique pseudo-plane. 

Consider a point $R$ of $\ell$ and consider the affine line $QR$. The point $P_\infty = QR \cap H_\infty$ is contained in a unique spread element $T \in \S$ different from $S$.
As $S$ is a normal element, the $(2n-1)$-space $\pi=\langle S,T\rangle$ intersects $\S$ in an $(n-1)$- spread $\S_0=\S\cap\pi$. Clearly, the line $\ell$ and the point $Q$ are contained in the affine translation plane $\T(\S_0)$ defined in the $2n$-space $\langle \pi, Q\rangle$. Hence, the unique pseudo-plane defined by $\ell$ and $Q$ is exactly $\T(\S_0)$. 
As this is true for any point $Q$ not incident with $\ell$, it follows that $\ell$ is a normal line.

Consider now a spread element $S \in \S$ that is {\it not} normal. Let $\ell$ be a line of $\L$ corresponding to an $n$-space containing $S$. 
As $S$ is not normal, there exist spread elements $T, U$ of $\S$  such that $\langle S,T\rangle \cap U \notin \{\emptyset,U\}$. 
Consider the unique $2n$-space $\Pi$ containing both the spread element $T$ and the $n$-space corresponding to $\ell$. Hence, $\Pi \cap H_\infty = \langle S,T\rangle$.

There exist affine points $R\in\ell$ and $Q \in \Pi$ such that $P_\infty=QR\cap H_\infty$ is contained in $T$. Consider the  unique pseudo-plane $\alpha$ defined by $\ell$ and $Q$. All affine points of the $n$-space $\langle Q, T \rangle$ are contained in $\alpha$ (as this $n$-space corresponds to the unique line of $\L$ defined by $Q$ and $R$). Clearly, all affine points of the $n$-space corresponding to $\ell$ are also contained in $\alpha$. It is now easy to see that all $q^{2n}$ affine points of $\Pi$ are contained in $\alpha$.

As the spread element $U$ intersects $\langle S,T \rangle $ non-empty, there exist affine points $R' \in \ell$ and $Q' \in \Pi$ such that the point $P'_\infty=R'Q'\cap H_\infty$ is contained in $U$.
Hence, all affine points of the $n$-space $\langle Q', U\rangle$ are contained in $\alpha$. As $U$ is not completely contained in $\langle S,T\rangle$, some of these affine points are not contained in $\Pi$. 

The pseudo-plane $\alpha$ contains more than $q^{2n}$ points, hence is not an affine plane. We conclude that $\ell$ is not a normal line.
\end{proof}


 
A set $N$ of normal lines through a common point $P$ is said to be in {\it general position} if for every line $\ell\in N$ and every subset $N'$ of $N\backslash \{l\}$ of size at most $r-1$, it holds that $\ell \cap \M(N') = \{P\}$, i.e.\ $\ell$ meets the linear manifold defined by $r-1$ of the other lines only in $P$. 

Note that for a set of normal lines through a common point in a translation Sperner space, the linear manifold they define is in fact a translation Sperner subspace.
\subsection{Results in terms of translation Sperner Spaces}

In this section we summarise our results, translated into the language of Sperner Spaces. Given Theorem \ref{thm:spernernormal}, the results follow immediately from Theorem \ref{r+1normal}, Theorem \ref{SpreadBySemifield}, Corollary \ref{4elementsDesarguesian} and Theorem \ref{r+1normalnotgeneral}.

\begin{corollary}
\label{r+1normalSperner}
Consider a translation Sperner space $\T$ with parameters $(q^{rn},q^n)$, $r>2$. 

If $\T$ contains $r+1$ normal lines through a common point in general position, then $\T$ is isomorphic to the affine space $\AG(r,q^n)$.

Suppose $r=3$ and $q$ odd. If $\S$ contains $3$ normal lines through a common point, contained in an affine plane of $\T$, then any plane containing exactly one of these lines is in fact an affine semifield plane. If $\T$ contains four normal lines through a common point, not all contained in a common affine plane, then $\T$ is isomorphic to the affine space $\AG(3,q^n)$.

Suppose $q$ odd. If $\T$ contains at least $r+1$ normal elements through a common point, such that the linear manifold defined by some $r$ of them is equal to $\T$, then $\T$ is is isomorphic to the affine space $\AG(r,q^n)$.
\end{corollary}



\end{document}